\documentclass{amsart}
\usepackage{amsmath,amsfonts,amsthm,amssymb,amsrefs}
\usepackage{tikz}
\usetikzlibrary{arrows}

\newtheorem{theorem}{Theorem}
\newtheorem{cor}[theorem]{Corollary}

\newtheorem{prop}[theorem]{Proposition}

\begin{document}
\title[Cusp excursions for the earthquake flow]{Cusp excursions for the earthquake flow on the once-punctured torus}
\date{}
\author{Ser-Wei Fu}

\begin{abstract}
In this paper we study the typical speed of a generic earthquake trajectory leaving compact sets in the moduli space of the once-punctured torus. Mirzakhani showed that the earthquake flow is measurably equivalent to the horocyclic flow, which has been studied extensively. Our main result shows that the earthquake flow and the horocyclic flow behave very differently in cusp excursions. In particular, we prove a relation between the systole function and continued fractions and discuss the cusp excursions of earthquake trajectories.
\end{abstract}

\maketitle

\section{Introduction}\label{INTRO}

Throughout the paper we will let $S=S_{1,1}$ be the once-punctured torus, $\mathcal{M}(S)$ be the moduli space, and $\mathcal{T}(S)$ be the Teichm\"{u}ller space. The motivation of the paper comes from the result of Mirzakhani in \cite{Mir} showing that the earthquake flow is ergodic with respect to the natural finite measure $\nu(S)$ coming from the Weil-Petersson measure on the bundle $\mathcal{P}^1\mathcal{M}(S)$ of unit length geodesic measured laminations. We use $E_t(X,\lambda)$ to denote the earthquake flow for time $t$ along the measured lamination $\lambda$ and $\ell_{sys}(X,\lambda)$ to be the length of the systole of $X$ (See Section~\ref{BG} for definitions). The systole length $\ell_{sys}(X,\lambda)$ is the same as $\ell_{sys}(X)$, with the prior in the bundle context and the latter in $\mathcal{M}(S)$.

The study of cusp excursions for the earthquake flow is aimed at understanding the constant $C_f(X,\lambda)$ defined with respect to a continuous increasing function $f$ by
\[
C_f(X,\lambda) = \limsup_{t\to\infty} \frac{-\log \ell_{sys}(E_t(X,\lambda))}{f(t)}.
\]
Ergodicity implies that $C_f(X,\lambda)$ is a constant for almost every $([X],\lambda)$ with respect to $\nu(S)$. The goal is to find the threshold function $f$ such that $C_f$ is a nonzero number that depends only on the surface.

One of the classical approaches to cusp excursion is the logarithmic laws initiated by Sullivan \cite{Sul}. The Borel-Cantelli lemma and a volume estimate of the thin part of the moduli space can be used to obtain a discrete upper bound for $C_{\log(t)}$. More specifically, for a general surface $S_g$ if there exists a continuous function $V$ such that the measure of the $\varepsilon$-thin part of $\mathcal{P}^1\mathcal{M}(S_{g})$ is $O(V(\varepsilon))$, then we obtain 
\[
\limsup_{n\to\infty} \frac{-\log \ell_{sys}(E_{n}(X,\lambda))}{\log n} \leq \lim_{\varepsilon \to 0}\frac{\log(V(\varepsilon))}{\log \varepsilon}, \text{ where } n\in\mathbb{N},
\]
for almost every $([X],\lambda) \in \mathcal{P}^1\mathcal{M}(S_{g})$ with respect to $\nu(S)$, see Section~\ref{UPPER}.

In \cite{Mir}, Mirzakhani showed that the earthquake flow is measurably equivalent to the horocyclic flow, which led to the ergodicity of the earthquake flow. The logarithm law of the horocyclic flow by Athreya in \cite{Ath} states that the discrete upper bound above is actually an equality even in the continuous case. In other words, the constant $C_{\log(t)}$ for the horocyclic flow is equal to 1/2 almost everywhere. However, the earthquake flow fails to satisfy the same property. Our main result shows that $C_{\sqrt{t}}$ is infinity almost everywhere for the once-punctured torus. 

\begin{theorem}\label{LEMMA}
Let $S$ be the once-punctured torus. For almost every $([X],\lambda) \in \mathcal{P}^1\mathcal{M}(S)$ with respect to $\nu(S)$, we have
\[
\limsup_{t\to\infty} \frac{-\log \ell_{sys}(E_t(X,\lambda))}{\sqrt{t}} = \infty.
\]
\end{theorem}
\newtheorem*{mainlemma}{Theorem~\ref{LEMMA}} 

The theorem echoes the result shown by Minsky and Weiss \cite[Proposition~8.1]{MinWei} stating that in the case of the once-punctured torus, the trajectories of the horocyclic flow and the earthquake flow can be infinitely far apart. As a corollary of the proof of the main theorem, we identify a family of earthquake trajectories associated to the golden ratio that goes arbitrarily deep into the thin part with the slowest possible speed.

We study the systole function along earthquake trajectories by explicitly parametrizing the space of hyperbolic metrics, the simple closed geodesics, and the space of measured laminations on the once-punctured torus. The systole functions on earthquake trajectories are closely related to the continued fraction expansion of \emph{slopes} that depends on $([X],\lambda)$, see Section~\ref{Case}. We show the result of the classical argument for cusp excursions in Section~\ref{UPPER} for completeness and prove the main theorem in Section~\ref{PROOF}. The proof is a straight forward computation of the limit supremum for any given $(X,\lambda)$ and the measure of the set that satisfies our statement. 

\noindent\textbf{Acknowledgement.} The author thanks Jayadev Athreya for suggesting the project as well as many useful conversations. The author is also grateful for discussions with Christopher Leininger, Matthew Stover, and Joseph Vandehey.


\section{Background}\label{BG}

We use a similar choice of notation as in \cite{Mir}. Let $\mathcal{T}(S)$ be the Teichm\"{u}ller space, which is the space of hyperbolic metrics on $S$. The mapping class group $\operatorname{Mod}(S)$, the group of automorphisms of $S$ up to isotopy, acts on $\mathcal{T}(S)$ and the quotient is the moduli space $\mathcal{M}(S)$. The \textit{Weil-Petersson symplectic form} induces a finite measure on $\mathcal{M}(S)$. 

Let $\mathcal{ML}(S)$ denote the space of \textit{measured laminations} on $S$ (see \cite{FaLaPo}). We use notation similar to \cite{Mir} and a thorough description of properties of $\mathcal{ML}(S)$ can be found in \cite{FaLaPo}. The space $\mathcal{ML}(S)$ carries a mapping class group invariant volume form $\mu_{Th}$ introduced by Thurston. Given $X \in \mathcal{T}(S)$ and $\lambda \in \mathcal{ML}(S)$, we denote the length of the lamination $\lambda$ in the hyperbolic metric $X$ by $\ell_\lambda(X)$. Finally we let $\mathcal{PT}(S) = \mathcal{T}(S) \times \mathcal{ML}(S)$ be the bundle of geodesic measured lamination over $\mathcal{T}(S)$, $\mathcal{P}^1\mathcal{T}(S)$ be the unit sub-bundle with respect to the hyperbolic length of $\lambda$, and $\mathcal{P}^1\mathcal{M}(S)$ be the quotient $\mathcal{P}^1\mathcal{T}(S) / \operatorname{Mod}(S)$. The earthquake-invariant measure on $\mathcal{P}^1\mathcal{M}(S)$ considered by Mirzakhani in \cite{Mir} is a finite measure denoted by $\nu(S)$.

The \emph{earthquake map} is a generalization of twisting along simple closed geodesics with respect to a hyperbolic metric on $S$, see \cites{Ker,Thu}. An earthquake along a simple closed curve, seen as a measured lamination, is defined to be a left twist of hyperbolic length 1. An earthquake along a measured lamination is defined to be the limit of the earthquake maps along a sequence of weight simple closed curves that converges to the measured lamination. The earthquake map is independent of the sequence you choose \cites{Ker,Thu} and can be extended into a flow \cite{Mir}. We denote the earthquake flow at time $t$ by $E_t(X,\lambda) \in \mathcal{PT}(S)$, where $E_t(X,\lambda) = (Y,\lambda)$ with $Y$ being the hyperbolic metric obtained by twisting $X$ along $\lambda$ for distance $t$.

The earthquake flow on $\mathcal{P}^1\mathcal{M}(S)$ is ergodic with respect to $\nu(S)$. This implies that almost every trajectory of the earthquake flow in $\mathcal{M}(S)$ is dense. Alternatively, we consider the systole length function $\ell_{sys}$ that maps $(X,\lambda) \in \mathcal{PT}(S)$ to the length of the shortest closed geodesic with respect to $X$. The ergodic property implies that
\[
\liminf_{t\to\infty} \ell_{sys} (E_t(X,\lambda)) = 0 \text{ for almost every } (X,\lambda) \in \mathcal{PT}(S).
\]

The Teichm\"{u}ller space of the once-punctured torus $\mathcal{T}(S_{1,1})$ is isometric to $\mathbb{H}^2$ and is parametrized via the \textit{Fenchel-Nielsen coordinate}. We fix a pair of generators for the fundamental group of $S_{1,1}$ that intersect once. Then every oriented simple closed geodesic corresponds to a pair of integers given by the algebraic intersection number with the generators. We denote an oriented simple closed geodesic by $\gamma(a,b)$, where $\operatorname{gcd}(a,b)=1$. 

The space of measured laminations $\mathcal{ML}(S_{1,1})$ is parametrized by the \textit{slope} (the ratio of $a$ and $b$) and the hyperbolic length with respect to $X$. We say that a measured lamination is rational if it is a weighted simple closed geodesic and irrational otherwise.

Let us setup the Fenchel-Nielsen coordinate. For $X \in \mathcal{T}(S_{1,1})$ and $\gamma = \gamma(a,b)$ a simple closed geodesic, the Fenchel-Nielsen coordinate of $X$ with respect to $\gamma$ is $(\ell_\gamma(X), \tau_\gamma(X))$, where $\tau_\gamma(X)$ is the twist parameter. We choose the twist parameter in the following manner. Let $\gamma(c,d)$ be an oriented simple closed geodesic such that $ad-bc=1$. The twist parameter $\tau_\gamma(X)=0$ if and only if $X$ is the hyperbolic metric that minimizes the length of $\gamma(c,d)$ for the fixed length $\ell_\gamma(X)$. The twist parameter is normalized such that integer valued twist parameters correspond to metrics obtained by Dehn twists about $\gamma$. The twist parameter $\tau_\gamma$ increases linearly along the earthquake trajectory of $E_t(X,\gamma)$.


\section{The punctured torus}\label{Case}

In this section we prove that the systole length function on a earthquake trajectory can be approximated via continued fraction expansions of irrational slopes.

\begin{prop}\label{estimate}
Fix a simple closed geodesic $\gamma$. If $X \in \mathcal{T}(S_{1,1})$ has Fenchel-Nielsen coordinate $(L,p/q)$ with respect to $\gamma$, where $\operatorname{gcd}(p,q)=1$ and $p/q\in(0,1)$, then 
\[
C_1(L) e^{-L/(2q)} < \ell_{sys} (X) < C_2(L) e^{-L/(2q)},
\]
where $C_1(L),C_2(L)$ both limits to $4$ when $p,q$ are fixed and $L$ goes to $\infty$.
\end{prop}
\newtheorem*{estimateprop}{Proposition~\ref{estimate}} 

\begin{proof}
We consider a once-punctured torus with a hyperbolic metric $X \in \mathcal{T}(S_{1,1})$. Figure~\ref{domain} below shows a fundamental domain in the upper half plane that is chosen canonically with a given Fenchel-Nielsen coordinate $(L,p/q)$ with respect to $\gamma$. 

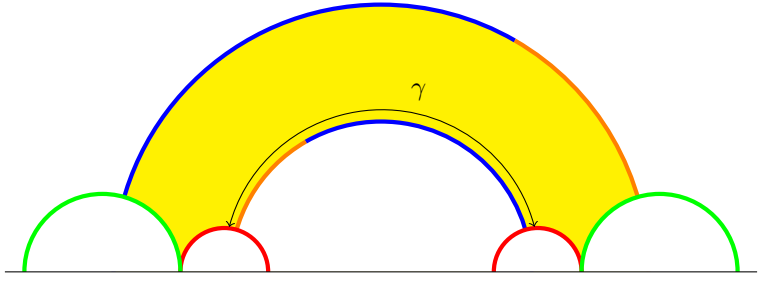
\begin{figure}[ht]
\begin{tikzpicture}[scale=2]

\draw[fill,yellow] (1.775,0) arc (0:180:1.775);
\draw[ultra thick,orange] (1.77778,0) arc (0:60:1.77778);
\draw[ultra thick,blue] (-1.77778,0) arc (180:60:1.77778);

\draw[fill,white] (0.99,0) arc (0:180:0.99);
\draw[ultra thick,blue] (1,0) arc (0:120:1);
\draw[ultra thick,orange] (-1,0) arc (180:120:1);

\draw[fill,white] (1.33333,0) arc (0:180:0.291666);
\draw[fill,white] (1.33333,0) arc (180:0:0.518518);
\draw[fill,white] (-1.33333,0) arc (180:0:0.291666);
\draw[fill,white] (-1.33333,0) arc (0:180:0.518518);
\draw[ultra thick,red] (1.33333,0) arc (0:180:0.291666);
\draw[ultra thick,green] (1.33333,0) arc (180:0:0.518518);
\draw[ultra thick,red] (-1.33333,0) arc (180:0:0.291666);
\draw[ultra thick,green] (-1.33333,0) arc (0:180:0.518518);

\draw (-2.5,0) -- (2.5,0);

\draw (0.25,1.2) node{$\gamma$};
\draw[<->] (1.02,0.3) arc (15:165:1.05);

\end{tikzpicture}
\caption{A fundamental domain for $X$ represented by $(L,p/q)$.}
\label{domain}
\end{figure}

The edges are identified by colors via hyperbolic isometries. The length ratio between the blue and the orange edges is equal to $p:q-p$. We use $R(L)$ to denote the length of the shortest geodesic arc with endpoints on $\gamma$. It is well known that $(R\circ R)(L)=L$ and 
\[
R(L) = 2\log \coth(L/4) = 2\log \left(\frac{e^{L/2}+1}{e^{L/2}-1}\right).
\]

For fixed $p/q \in (0,1)$, the systole function is realized by a simple closed curve that intersects $\gamma$ more than once when $L = \ell_\gamma(X)$ is large enough. We use the analogy from the flat torus case and look at the length of simple closed curves $\alpha$ that intersect $\gamma$ exactly $q$ times as seen in Figure~\ref{analogy}.

\begin{figure}[ht]
\begin{tikzpicture}[scale=2.5]

\draw[ultra thick,blue] (1,0) -- (3.5,0);
\draw[ultra thick,orange] (0,0) -- (1,0);
\draw[ultra thick,orange] (3.5,0.5) -- (2.5,0.5);
\draw[ultra thick,blue] (2.5,0.5) -- (0,0.5);
\draw (2,0.05) node{$\gamma$};
\draw[ultra thick,red] (0,0) -- (0,0.5);
\draw[ultra thick,red] (3.5,0) -- (3.5,0.5);

\draw[dotted] (1.75,0) -- (1.75,0.5);
\draw[dotted] (2.75,0) -- (2.75,0.5);
\draw[dotted] (0.25,0) -- (0.25,0.5);
\draw[dotted] (1.25,0) -- (1.25,0.5);
\draw[dotted] (2.25,0) -- (2.25,0.5);
\draw[dotted] (3.25,0) -- (3.25,0.5);
\draw[dotted] (0.75,0) -- (0.75,0.5);
\draw (1.7,0.25) node{$\alpha$};

\end{tikzpicture}
\caption{The flat torus analogy.}
\label{analogy}
\end{figure}
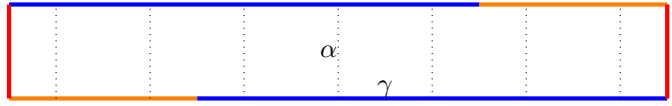

We obtain the following inequality involving $a$, the length of $\alpha$, with respect to $L$ which came from switching the role of $\gamma$ and $\alpha$ in the pictures. 
\[
q R(a) < L < q R(a) + \frac{qa}{2}.
\]
Since $\gamma$ consists of $q$ geodesic segments from $\alpha$ to $\alpha$, the left-hand-side is obtained from $R(a)$. The right-hand-side is obtained by considering a piecewise geodesic homotopic to $\gamma$. The piecewise geodesic goes along the geodesic segment that realizes the injectivity radius of $\alpha$ $q$ times and travel along $\alpha$ otherwise. 

Reorganizing the terms will get us
\[
e^{-a/4}\tanh(a/4) < e^{-L/(2q)} < \tanh(a/4).
\]
A closed curve homotopic to $\alpha$ is obtained by concatenating $q$ radial segments between the two circles in Figure~\ref{domain}. The hyperbolic length of the radial segments goes to zero as $L \to \infty$ since the Euclidean length of the segments goes to zero while being bounded distance away from the $x$-axis. Hence $a \to 0$ as $L\to\infty$, which implies that $e^{-a/4}\tanh(a/4)$ and $\tanh(a/4)$ are both approximately $a/4$. In other words, there exists $C_1(L), C_2(L)$ that both converges to $4$ such that
\[
C_1(L) e^{-L/(2q)} < a < C_2(L) e^{-L/(2q)}.
\]
To complete the proof, we need to show that the length of $\alpha$ is shorter than any other simple closed geodesic for large $L$. In fact, the geodesic homotopic to $\alpha$ is the only simple closed geodesic whose length goes to zero as $L\to\infty$. Therefore the estimate holds for the systole length.
\end{proof}

The estimate given by Proposition~\ref{estimate} allows us to construct explicit sequence of sample points along the earthquake trajectory. The sequence of sample points will be used in the proof of the main theorem for a lower bound of $C_{\sqrt{t}}(X,\lambda)$.

\begin{prop}\label{ZERO}
For any $X \in \mathcal{T}(S_{1,1})$ and any irrational $\lambda \in \mathcal{ML}(S_{1,1})$, we have
\[
\liminf_{t\to\infty} \ell_{sys}E_t(X,\lambda) = 0.
\]
\end{prop}

\begin{proof}
It suffices to construct a sequence $\{t_n\}$ such that $\ell_{sys}E_{t_n}(X,\lambda)$ goes to zero as $n\to\infty$. Without loss of generality, we assume that $\ell_\lambda(X) = 1$. There exists a sequence of simple closed geodesics $\{\gamma_n\}$ such that $( \ell_{\gamma_n}(X))^{-1} \gamma_n  \to \lambda$ in $\mathcal{ML}(S_{1,1})$. Fix $\varepsilon > 0$ and $T > 0$. There exists $N=N(T,\varepsilon)$ such that $n \geq N$ implies the following.
\[
\left| \ell_{sys}E_t(X,\lambda) -\ell_{sys}E_t \left( X, \frac{\gamma_n}{\ell_{\gamma_n}(X)} \right) \right| < \varepsilon \text{ for all } 0<t<T.
\]

Let $\alpha$ be a closed geodesic that realizes $\ell_{sys}(X,\lambda)$ and $\gamma = \gamma_m$ for some $m \geq N$. We will assume that $m$ is large enough from here on, hence the length of $\gamma$ will also be large enough. We observe that there exists a pair $p_0,q_0$ such that $\alpha$ is the systole of $X_0 = (\ell_\gamma(X),p_0/q_0)$ in Fenchel-Nielsen coordinate with respect to $\gamma$. Let $X_1$ be $(\ell_\gamma(X),p_1/q_1)$ with $p_1/q_1$ be the least rational number larger than $p_0/q_0$ with $1 \leq q_1 \leq q_0$. We iteratively define $X_j$ from $X_{j-1}$ similarly and choose $t_j$ to correspond with $X_j$ on the earthquake flow. Let $k$ be such that $t_k \leq T \leq t_{k+1}$.

The choice of $X_j$'s along with the estimates in Proposition~\ref{estimate} imply that the finite sequence
\[
\{ \ell_{sys} E_{t_j}(X,( \ell_{\gamma}(X))^{-1} \gamma) \}_{j=1,\ldots,k}
\]
is decreasing. Therefore 
\[
\inf_{0<t<T} \ell_{sys}E_t(X,\lambda) < (4+\varepsilon) e^{-\ell_{\gamma}(X)/(2q_k)}.
\]
The value of $q_k$ uniquely depends on $p_0$, $q_0$, and $k$. Since the length of $\alpha$ is fixed, we use Proposition~\ref{estimate} to conclude that the ratio $\ell_{\gamma}(X) / q_0$ as a function of $m$ converges to a constant. Hence it suffices to consider the ratio $q_0/q_k$ and show that it goes to infinity as $m$ goes to infinity. 

We use some properties of continued fractions from \cite{Khi} in the rest of the proof. By construction, for $q_j\neq 1$, $p_j/q_j$ is a convergent of $p_0/q_0$. To be more specific, if $p_0/q_0$ is $[a_0; a_1, a_2, \ldots, a_n]$ in the usual convention of continued fraction expansion, then $p_j/q_j$ is $[a_0; a_1, a_2, \ldots, a_{h(j)}]$, where $h(j)$ is equal to $n-2j$ or $n-2j+1$ since only the odd convergents are greater than $p_0/q_0$.

The ratio $q_0/q_k$ depends on the continued fraction expansion of $p_0/q_0$. In particular, the ratio grows with $k$ unless $q_k=1$. We observe that $q_0/q_k$ goes to infinity as $k$ goes to infinity whenever $p_0/q_0$ converges to an irrational number as $m$ increases, which is a consequence of $\lambda$ being irrational. Therefore for any fixed $\varepsilon$, as $T$ increases, $m$ and $k$ go to infinity and that completes the proof.
\end{proof}


\section{Discrete upper bound}\label{UPPER}

We include a discussion on the classical upper bound obtained by the Borel-Cantelli lemma for completion. The volume estimate of the thin-part for general surfaces $S_g$ can be obtained through Mirzakhani's generalized McShane's identity \cite{Mir2} which is beyond the scope of this paper. For the next proposition we need

\begin{theorem}[McShane identity]\label{McShane}
Let $X$ be a point in the Teichm\"{u}ller space of the once-punctured torus. Then we have
\[
\sum_{\gamma} \left( 1+e^{\ell_\gamma(X)} \right)^{-1} = \frac{1}{2},
\]
where $\ell_\gamma(X)$ is the length of the geodesic $\gamma$ with respect to $X$ and the sum is over all simple closed geodesics $\gamma$ on $X$.
\end{theorem}

\begin{prop}\label{THIN}
Let $\nu(S_{1,1})$ be the finite measure on $\mathcal{P}^1\mathcal{M}(S_{1,1})$. Then
\[
\nu(S_{1,1})\{ ([X],\lambda) \in \mathcal{P}^1\mathcal{M}(S_{1,1}) \mid \ell_{sys}(X,\lambda) < \varepsilon \} = O\left(\frac{\varepsilon}{\log \varepsilon}\right) \text{ as } \varepsilon\to 0.
\]
\end{prop}

\begin{proof}
Recall the computation from \cite{Mir2} showing how the volume of $\mathcal{M}_{1,1}$ is computed using Theorem~\ref{McShane} in \cite{McS}.

\[
\operatorname{Vol}(\mathcal{M}_{1,1}) = 2\int_0^\infty \ell f(\ell)\ d\ell = 2 \int_0^\infty \frac{\ell}{1+e^\ell}\ d\ell = \frac{\pi^2}{6}.
\]

In \cite{Mir} Mirzakhani showed that the earthquake-invariant measure $\nu$ restricted to the space of unit-length measured laminations is equal to the Thurston volume of the unit ball in the space of measured laminations. Therefore we can compute the volume of the $\varepsilon$-thin part of $\mathcal{P}^1\mathcal{M}(S_{1,1})$ with the help of the following claim.

\noindent \textbf{Claim.} There exist constants $C_1,C_2 > 0$ such that for any $X \in \mathcal{T}(S_{1,1})$ with $\ell_{sys}(X)=\varepsilon$, 
\[
\frac{-C_1}{\varepsilon\log\varepsilon} \leq \mu_{Th} \{ \lambda \in \mathcal{ML}(S_{1,1}) \mid \ell_\lambda(X) \leq 1 \} \leq \frac{-C_2}{\varepsilon\log\varepsilon}.
\]

\noindent \textit{Proof of Claim.}
Let $\gamma$ be a systole of $X$. The Fenchel-Nielsen coordinate of $X$ with respect to $\gamma$ is $(\varepsilon, \tau)$ for some $\tau$. Let $\alpha$ be a simple closed curve that intersects $\gamma$ once. Then $\lambda \in \mathcal{ML}(S)$ has coordinate $(\iota(\lambda,\gamma), \iota(\lambda,\alpha))$ where $\iota(\cdot,\cdot)$ is the algebraic intersection number with some fixed orientation.

The Thurston volume can be computed by a double integral of the characteristic function of the unit ball in $\mathcal{ML}(S_{1,1})$. Since the length of $\lambda$ is linear with respect to its coordinate, we can rewrite the integral in polar coordinates.
\[
\mu_{Th} \{ \lambda \in \mathcal{ML}(S) \mid \ell_\lambda(X) \leq 1 \} = \int_0^{\pi} \int_0^{r(\theta)} rdrd\theta,
\]
where $r(\theta)$ is the reciprocal of the length of the measured lamination with coordinate $(\cos\theta, \sin\theta)$.

We can choose $\alpha$ such that $\ell_\alpha(X)$ is bounded between $R(\varepsilon)$ and $R(\varepsilon)+\varepsilon$, where $R(\varepsilon)$ is as defined in the proof of Proposition~\ref{estimate}. The length of the measured lamination $(\cos\theta, \sin\theta)$ is bounded between $\max\{R(\varepsilon)|\cos\theta|, \varepsilon|\sin\theta|\}$ and $(R(\varepsilon)|\cos\theta| + \varepsilon|\sin\theta|)$. The integral 
\[
\int_0^{\pi} \frac{1}{(R(\varepsilon)|\cos\theta| + \varepsilon|\sin\theta|)^2}d\theta
\]
is estimated by the standard method of breaking it into pieces at the points when $R(\varepsilon)|\cos\theta| = \varepsilon|\sin\theta|$. Finally we estimate $R(\varepsilon)$ by $\log(1/\varepsilon)$ when $\varepsilon$ is small.

That finishes the proof of the claim. Finally, we consider the integral below and standard computation will complete the proof.
\[
\int_0^\varepsilon \frac{\ell}{1+e^\ell} \cdot \frac{-1}{\ell \log \ell} \ d\ell = O\left(\frac{\varepsilon}{\log \varepsilon}\right) \text{ as } \varepsilon\to 0.
\]
\end{proof}

The following proposition is the result of a standard argument using the Borel-Cantelli lemma along with the volume of the thin part as stated in Proposition~\ref{THIN}. Recall the easier part of the Borel-Cantelli lemma.

\begin{theorem}[Borel-Cantelli lemma]\label{BClemma}
Let $(\mathfrak{X},\mathcal{F},\mu)$ be a probability space, where $\mathfrak{X}$ is the sample space, $\mathcal{F}$ is a set of events, and $\mu$ is a probability measure. Let $X_n : \mathfrak{X} \to \{0,1\}$ be a sequence of random variables with $\mu(x\in \mathfrak{X} : X_n(x)=1) = p_n$.
\[
\text{If } \sum_{n=1}^\infty p_n < \infty, \text{ then } \mu\left( \sum_{n=1}^\infty X_n = \infty \right) = 0.
\]
\end{theorem}

\begin{prop}\label{UPPERB}
Let $S = S_{1,1}$ and $t(n)$ be any fixed sequence of real numbers. For almost every $(X,\lambda) \in \mathcal{T}(S) \times \mathcal{MF}(S)$, we have
\[
\limsup_{n\to\infty} \frac{-\log \ell_{sys}(E_{t(n)}(X,\lambda))}{\log n} \leq 1,
\]
where $\ell_{sys}$ is the systole function and $E_{t(n)}$ is the earthquake function for time $t(n)$.
\end{prop}

\begin{proof}
We use the setting where $\mathfrak{X}$ is $\mathcal{P}^1\mathcal{M}(S_{1,1})$, $\mu$ is $\nu(S_{1,1})$, and $\mathcal{F}$ is the set of measurable sets. Let $\{ Y_n \}$ be random variables defined by $-\log \ell_{sys}(E_{t(n)}(\ \cdot\ ))$. The value $\mu(Y_n > -\log \varepsilon)$ stands for the measure of the set of points $([X],\lambda)$ in $\mathcal{P}^1\mathcal{M}(S_{1,1})$ such that the systole length is shorter than $\varepsilon$ after earthquaking by time $t(n)$ in the direction $\lambda$. Using the fact that the measure $\nu$ is invariant under earthquake flows and Proposition~\ref{THIN}, 
\[
\mu(Y_n > -\log \varepsilon) = O\left( \frac{\varepsilon}{\log \varepsilon} \right) \text{ as } \varepsilon\to 0.
\]
Let $\{r_n\}$ be a sequence of positive real numbers and
\[
X_n = \begin{cases} 1 & Y_n > r_n, \\ 0 & \text{otherwise}. \end{cases}
\]
Our goal is to show that the limsup of $Y_n/\log n$ is bounded above. By the definition of limsup, it suffices to show that for all $\delta>0$ and almost every $([X],\lambda)$, $Y_n > (1+\delta)\log n$ for only a finite number of $n$. This follows from setting $r_n = (1+\delta)\log n$ and applying Theorem~\ref{BClemma}. A quick computation shows that
\[
\sum_{n=1}^\infty p_n = \sum_{n=1}^\infty \mu\left(Y_n > (1+\delta)\log n\right) = \sum_{n=1}^\infty O\left(\frac{n^{-1-\delta}}{(1+\delta)\log n}\right) < \infty.
\]
Hence $\mu\left( \sum_{n=1}^\infty X_n = \infty \right) = 0$, which is equivalent to $\mu\left( \sum_{n=1}^\infty X_n < \infty \right) = 1$, and this implies that the limsup of $Y_n/\log n$ is bounded above by 1 almost everywhere.
\end{proof}

\noindent \textbf{Remark.} It is important to note that $t(n)$ is a sequence chosen prior to the almost everywhere property. For example, a constant sequence would satisfy the upper bound trivially for every $([X],\lambda)$. The choice of $t(n) = n$ were used in classical logarithmic law papers \cites{Sul,Mas} and we will show that the same choice does not hold for the earthquake flow.


\section{Main Theorem}\label{PROOF}

We now prove the main theorem. The general notation is as in the proof of Proposition~\ref{ZERO}.

\begin{mainlemma}
Let $S$ be the once-punctured torus or the four-punctured sphere. For almost every $([X],\lambda) \in \mathcal{P}^1\mathcal{M}(S)$ with respect to $\nu(S)$, we have
\[
\limsup_{t\to\infty} \frac{-\log \ell_{sys}(E_t(X,\lambda))}{\sqrt{t}} = \infty.
\]
\end{mainlemma}

\begin{proof}
We consider $\lambda$ with irrational slope and a sequence $\{\gamma_n\}$ with $( \ell_{\gamma_n}(X))^{-1} \gamma_n  \to \lambda$. Let $\alpha$ be a closed geodesic that realizes $\ell_{sys}(X,\lambda)$ and $\gamma = \gamma_m$ for some $m \geq N$ large enough. 

Recall from the proof of Proposition~\ref{ZERO} that for each fixed $\varepsilon>0$, $T>0$, and $m>N(T,\varepsilon)$, there exists a finite sequence $\{X_j\}_{j=0,\ldots,k}$ that are candidates along $E_t(X,\lambda), 0<t<T,$ that have small systole length. The candidates are described by their Fenchel-Nielsen coordinates with respect to $\gamma_m$ that corresponds with the sequence of rational twists $\{p_j/q_j\}_{j=0,\ldots,k}$ arising as convergents. Finally for each rational twist there is a corresponding time, hence we have a sequence $\{t_j\}_{j=0,\ldots,k}$. 

Proposition~\ref{estimate} shows that $-\log \ell_{sys}(E_{t_j}(X,\lambda)) $ is approximately $\ell_{\gamma}(X) / (2q_j)$. To finish the proof it suffices to estimate $t_j$ for almost every $(X,\lambda)$. For the earthquake flow $E_{t}(X,( \ell_{\gamma}(X))^{-1} \gamma)$, we have 
\[
t_j-t_0 = \left(\frac{p_j}{q_j} - \frac{p_0}{q_0}\right) (\ell_\gamma(X))^2,
\] 
where $t_0$, corresponding to $X_0$, could be negative. The factor of $(\ell_\gamma(X))^2$ is the result of our normalization of the twist parameter and the weight on $\gamma$.

The value of $t_0$ is bounded by half the maximum time that $\alpha$ realizes the systole length along the earthquake flow. We use $t_1-t_0$ as an upper bound for $|t_0|$. Recall that by construction, for $q_j\neq 1$, $p_j/q_j$ is a convergent of $p_0/q_0$. To be more specific, if $p_0/q_0$ is $[a_0; a_1, a_2, \ldots, a_n]$ in the usual convention of continued fraction expansion, then $p_j/q_j$ is $[a_0; a_1, a_2, \ldots, a_{h(j)}]$, where $h(j)$ is equal to $n-2j$ or $n-2j+1$ since only the odd convergents are greater than $p_0/q_0$. Again using properties of continued fractions from \cite{Khi} and notation from the proof of Proposition~\ref{ZERO}, we have the following bound that compares convergents.
\[
\frac{p_j}{q_j} - \frac{p_0}{q_0} < \frac{1}{a_{h(j)+1}q_j^2}.
\]

Combining our approximations above, first we have 
\[
\frac{-\log \ell_{sys}(E_{t_j}(X,\lambda))}{\sqrt{t_j}} >  \left(\frac{\ell_\gamma(X)}{2q_j} - \varepsilon\right) t_j^{-1/2}.
\]
Next we use $t_j = (t_j - t_0) + t_0$ to obtain
\[
\left(\frac{\ell_\gamma(X)}{2q_j} - \varepsilon\right) t_j^{-1/2} > \left(\frac{\ell_\gamma(X)}{2q_j} - \varepsilon\right) \left[ \frac{(\ell_\gamma(X))^2}{a_{h(1)+1}q_1^2} + \frac{(\ell_\gamma(X))^2}{a_{h(j)+1}q_j^2}\right]^{-1/2}.
\]
We use the fact that $q_1$ and $\ell_\gamma(X)$ are much larger than $q_j$ to simplify
\[
\frac{-\log \ell_{sys}(E_{t_j}(X,\lambda))}{\sqrt{t_j}} > \left(\frac{1}{2} - \frac{\varepsilon q_j}{\ell_\gamma(X)}\right) \left[ \frac{1}{a_{h(1)+1}(q_1/q_j)^2} + \frac{1}{a_{h(j)+1}}\right]^{-1/2}
\]
into 
\[
\frac{-\log \ell_{sys}(E_{t_j}(X,\lambda))}{\sqrt{t_j}} > \frac{1}{2}\sqrt{a_{h(j)+1}} - \delta
\]
for $j$ large enough and for some small $\delta$. 

As $m$ goes to infinity, the slope of $\gamma_m$ with respect to the chosen basis converges to some irrational number $s$, the slope of $\lambda$. Recall that $p_0/q_0$ depends on $X$ and a choice of a systole curve $\alpha$. As a function of $m$, $p_0/q_0$ converges to some value $\sigma = \sigma(s)$. In fact, for any fixed $j$, the sequence $\{p_j/q_j\}$ converges to $\sigma$ as $m\to\infty$. This comes from the observation that any finite sequence $\{p_j(m)/q_j(m)\}_{j=1}^k$ is a finite subsequence of the convergents of $\sigma$ for $m$ large enough. 

By choosing $T$ arbitrarily large and $\varepsilon$ arbitrarily small, we get
\[
\limsup_{t\to\infty} \frac{-\log \ell_{sys}(E_{t}(X,\lambda))}{\sqrt{t}} > \limsup_{j\to\infty} \frac{1}{2}\sqrt{a_{h(j)+1}(\sigma)} - \varepsilon,
\]
where $a_n(\sigma)$ is the partial quotients of $\sigma$. The right-hand-side is infinity for almost every real number $\sigma$.

The final step is to analyze the function $\sigma : \mathbb{R}\to \mathbb{R}$ that depends only on $X$ and a choice of a systole curve $\alpha$. Rational slopes would be mapped to the rational twists in the sense of $p_0/q_0$ with respect to $\gamma_m$. To be explicit, suppose that $\gamma_m = \gamma(a,b)$, $\alpha = \gamma(c,d)$, and $\sigma(a/b) = p_0/q_0$. The value of $q_0$ is equal to the intersection number of $\gamma(a,b)$ and $\gamma(c,d)$, that is, $q_0 = |ad-bc|$. The value of $p_0$ depends on $X$ since it is related to the choice of twist parameter. The function $\sigma$ is a continuous bijection over the rational numbers and irrational numbers are mapped by continuity. 

For any integer $q$ we consider the set of integer pairs $(a,b)$ such that $|ad-bc|=q$ for $(c,d)$ fixed. This set is mapped to the set $p/q$ under $\sigma$. Hence the derivative of $\sigma$ is controlled by the choice of $(c,d)$, which is based on a choice of basis for $\pi_1(S_{1,1})$.

In an open neighborhood of a fixed $(X,\lambda)$, we can pick a suitable choice of basis for the fundamental group of $S_{1,1}$ so that $\sigma$ is an absolutely continuous function over the set of slopes for $\lambda$ in the neighborhood. This implies that sets of full measure are mapped to sets of full measure under $\sigma$. Therefore for almost every slope $s$, $\sigma(s)$ satisfies 
\[
\limsup_{j\to\infty} a_{h(j)+1}(\sigma(s)) = \infty.
\]

Therefore by applying Fubini's Theorem to the measure of the set of $([X],\lambda)$ such that $C_{\sqrt{t}} = \infty$, we obtain that 
\[
\limsup_{t\to\infty} \frac{-\log \ell_{sys}(E_t(X,\lambda))}{\sqrt{t}} = \infty
\]
for almost every $([X],\lambda) \in \mathcal{P}^1\mathcal{M}(S)$ with respect to $\nu(S)$.
\end{proof}

The special cases when $\sigma$ is equal to the golden ratio or when the tail of the continued fraction expansion is all $1$'s give us the corollary below.

\begin{cor}
For all $([X],\lambda) \in \mathcal{P}^1\mathcal{M}(S)$, 
\[
\limsup_{t\to\infty} \frac{-\log \ell_{sys}(E_t(X,\lambda))}{\sqrt{t}} \geq \frac{\varphi^{3}}{2\sqrt{\varphi^4-1}}, \ \varphi = \frac{1+\sqrt{5}}{2},
\]
if and only if $\lambda$ has irrational slope. Furthermore the inequality is sharp.
\end{cor}

\begin{proof}
If $\lambda$ has rational slope, then 
\[
\limsup_{t\to\infty} \big[ -\log \ell_{sys}(E_t(X,\lambda)) \big]
\]
is a constant that depends only on $\ell_\lambda(X)$. 

For irrational $\lambda$ we consider $\sigma(\lambda)$ and its continued fraction expansion $a_m/b_m$. The value of $C_{\sqrt{t}}$ is larger for faster growing $b_m$'s, i.e., faster decreasing $q_j$'s. We can focus on the case when $\sigma(\lambda)$ is equal to the golden ratio $\varphi$ since it is the irrational number with the slowest possible growth of $b_m$.

We estimate $C_{\sqrt{t}}$ similar to the steps in the proof of Theorem~\ref{LEMMA}. We use two properties of $\varphi$.
\[
\frac{p_j}{q_j} - \frac{p_0}{q_0}= \frac{1}{q_j \cdot q_{j-2}} + \frac{1}{q_{j-2} \cdot q_{j-4}} + \cdots
\]
where the last term is either $1/q_0q_2$ or $1/q_0q_1$. The ratio between values of $q_j$ is controlled.
\[
\lim_{m\to\infty} \frac{b_{m+c}}{b_m} = \varphi^c.
\]
Combining the approximations all together we have
\[
\frac{-\log \ell_{sys}(E_{t_j}(X,\lambda))}{\sqrt{t_j}} \approx \frac{\ell_{\gamma}(X) / (2q_j)}{\sqrt{(\ell_\gamma(X))^2\cdot \sum_{i=1}^{\lceil j/2\rceil} \frac{1}{q_j^2} \left( \varphi^{2-4i} \right)}} = \frac{\varphi}{2}\sqrt{\frac{1-\varphi^{-2j}}{1-\varphi^{-4}}}.
\]
As $q_j$'s and $j$ become larger, we get our desired result for $\sigma(\lambda) = \varphi$.
\[
\limsup_{t\to\infty} \frac{-\log \ell_{sys}(E_t(X,\lambda))}{\sqrt{t}} = \frac{\varphi^{3}}{2\sqrt{\varphi^4-1}}.
\]
\end{proof}


\end{document}